\documentclass[10pt]{article}
\usepackage{tikz, multirow, multicol}
\usetikzlibrary{decorations.markings}
\usetikzlibrary{arrows}
\tikzset{->-/.style={decoration={
  markings,
  mark=at position .5 with {\arrow[scale = 1.5]{stealth}}},postaction={decorate}}}

\usepackage{hyperref}

\usepackage{graphicx} 
\usepackage{amsthm,amsfonts, amssymb, amsmath,mathrsfs}
\newtheorem{theorem}{Theorem}[section]        
\newtheorem{lemma}[theorem]{Lemma}			  
\newtheorem{corollary}[theorem]{Corollary}    

\theoremstyle{definition}

\newtheorem{example}[theorem]{Example}

\newtheorem{remark}[theorem]{Remark}	
\newcommand{\B}{\mathcal{B}}
\newcommand{\A}{\mathcal{A}}
\newcommand{\cC}{\mathcal{C}}

\newcommand{\R}{\mathbb{R}}

\newcommand{\cH}{\mathcal{H}}

\newcommand{\G}{\mathcal{G}}

\newcommand{\cS}{\mathcal{S}}
\newcommand{\cT}{\mathcal{A}}

\newcommand{\ccP}{\mathscr{P}}
\renewcommand{\O}{{\rm{O}}}
\newcommand{\rr}{\mathbb{S}}

\usepackage{stackengine}
 
 \newcommand{\one}{\mathbf{1}}
 \def\Plus{\texttt{\raisebox{-1.48ex}{\scalebox{1}{+}}}}
\def\Minus{\texttt{\raisebox{-1.48ex}{\scalebox{1}{-}}}}
\newcommand\oplu{\stackMath\mathbin{\stackinset{c}{0ex}{c}{0ex}{\normalfont\Plus}{\bigcirc}}}
\newcommand\omin{\stackMath\mathbin{\stackinset{c}{0ex}{c}{0ex}{\normalfont\Minus}{\bigcirc}}}
 \newcommand\ostar{\stackMath\mathbin{\stackinset{c}{0ex}{c}{0ex}{\normalfont\ast}{\bigcirc}}}
  \newcommand{\sstar}{*}
 
\title{Sign patterns which require or allow the strong multiplicity property}

\author{
Abhilash Saha
\thanks{{Department of Mathematics and Statistics, Indian Institute of Science Education and Research Kolkata, Mohanpur, WB, 741 246, India}
(as21ms054@iiserkol.ac.in).
Research supported in part by MITACS GRI.}
\and
Leona Tilis
\thanks{Department of Mathematics and Statistics,
McMaster University, Hamilton, ON, L8S 4L8, Canada 
(tilisl@mcmaster.ca, vantuyla@mcmaster.ca).
Research supported in part by NSERC USRA (Tilis) and NSERC Discovery Grant 2024--05299 (Van Tuyl).}
\and
Kevin N.\ Vander Meulen
\thanks{Department of Mathematics, Redeemer University, Ancaster, ON, L9K 1J4, Canada
(kvanderm@redeemer.ca). Research supported in part by NSERC Discovery Grant 2022--05137.}
\and
{Adam Van Tuyl\footnotemark[2]}
}

\providecommand{\keywords}[1]{{\textit{Keywords and phrases:}} #1}
\providecommand{\subjclass}[1]{{\textit{AMS Subject [2020]:}} #1}
\date{\today}

\begin{document}
\maketitle 
\begin{abstract}
    We initiate a study of sign patterns that require or allow 
    the non-symmetric strong multiplicity property (nSMP). 
    We show that all cycle patterns require the nSMP, regardless of the 
    number of nonzero diagonal entries. We present a class of 
    Hessenberg patterns that require the nSMP. We characterize which star 
    sign patterns require, which allow, and which do not allow the nSMP. 
    We show that if a pattern requires distinct eigenvalues, then it requires 
    the nSMP. Further, we characterize the patterns that allow the nSMP as being 
    precisely the set of patterns that allow distinct eigenvalues, a property 
    that corresponds to 
    a simple feature 
    of the associated digraph.
    We also characterize the sign patterns of order at most three according to whether they require, allow, or do not allow the nSMP.
\end{abstract}

\medskip
\subjclass{15B35, 15A18, 05C50} 

\noindent
\keywords{Sign pattern matrices, Non-symmetric strong multiplicity property, Distinct eigenvalues, Digraphs}

\tikzstyle{place}=[circle,draw=black!100,fill=black!100,thick,inner sep=0pt,minimum size=1mm]
\tikzstyle{left}=[>=latex,<-,semithick]
\tikzstyle{right}=[>=latex,->,semithick]
\tikzstyle{nleft}=[>=latex,-,semithick]
\tikzstyle{nright}=[>=latex,-,semithick]
\tikzstyle{right2}=[-,semithick]

\section{Introduction}
A real $n\times n$ matrix $A$ has the \emph{non-symmetric strong multiplicity property} (nSMP) if the system with conditions
\begin{eqnarray}
    A\circ X&=&\O,\label{eq:had}\\ 
    AX^T-X^TA&=&\O,\qquad {\rm{and}}\label{eq:commute}\\
    {\rm{tr}}(X^TA^k)&=&0,\  \qquad {\rm{for}} \ \ 0\leq k<n, \label{eq:trace}
    \end{eqnarray}
has only the trivial solution $X=\O$, where $\circ$ represents the Hadamard (entrywise) product and $\O$ is the zero matrix. A matrix $A$ has the \emph{non-symmetric strong spectral property} (nSSP) if the two conditions (\ref{eq:had}) and (\ref{eq:commute}) are sufficient to guarantee that $X=\O$.
Consequently, if $A$ has the nSSP then $A$ has the nSMP.

The non-symmetric strong multiplicity property was developed in \cite{CGSV}, building upon previous work on the non-symmetric strong spectral property  in \cite{bifurcate}. If a pattern allows the nSMP and a particular matrix $A$ with the pattern has the nSMP, then there are  implications for the eigenvalue multiplicity lists for superpatterns of the pattern (see e.g.  Theorem~\ref{thm:bifurcation}). The nSMP can be helpful in characterizing the allow sequence of
distinct eigenvalues of various matrix sign patterns (see \cite{BBCCVV} and \cite{CGSV}). 
The nSMP was used in \cite{CCCV} to help characterize the allow sequence of distinct eigenvalues of various {nonzero} matrix patterns. A key issue for these types of problems is to be able to recognize if a particular matrix has the nSMP. In this paper, we lay the groundwork for determining which patterns require the nSMP, and which sign patterns do not allow the nSMP. We are able to characterize the latter, and provide classes of patterns that satisfy the former.

In Section~\ref{sec:def}, we describe relevant definitions and some intermediate  results regarding the nSMP. In Section~\ref{sec:hollow},
we observe that any pattern that has all  of its off-diagonal entries nonzero, will require the nSMP.
In Section~\ref{sec:cycle}, we characterize the sign patterns that require the nSMP among the patterns whose digraph is a cycle. 
In Section~\ref{sec:star},
we characterize star digraphs that represent a sign pattern that requires the nSMP.
   Section~\ref{sec:hessenberg} describes a family
of Hessenberg patterns that requires the nSMP.
In Section~\ref{sec:hessy}, we answer an open question from \cite{CGSV}, showing that every pattern that requires distinct eigenvalues also requires the nSMP. We also demonstrate that a pattern allows distinct eigenvalues if and only if it allows the nSMP. We note that such patterns are simple to recognize from their digraph, based on the maximum size of a largest composite cycle.
In Section~\ref{sec:allows}, we describe some constructions of patterns that allow but do not require the nSMP. We also observe that while every spectrally arbitrary pattern allows the nSMP, there exists some spectrally arbitrary patterns that do not require the nSMP.
Then, in Section~\ref{sec:small}, we characterize small order patterns according to the nSMP and make some observations about reducible patterns.
We finish with some concluding comments and questions in Section~\ref{sec:conclude}.


\section{Terms and preliminary observations}\label{sec:def}

A \emph{sign pattern} $\mathcal{A}$ is an $n\times n$ matrix with entries in $\{+,-,0\}$. 
The \emph{qualitative class} of
$\cT$, denoted $Q(\cT)$, is the set of all
real matrices $A$ such that $A_{ij}>0$ if $\cT_{ij}=+$, $A_{ij}<0$ if $\cT_{ij}=-$, and $A_{ij}=0$ if $\cT_{ij}=0$. A \emph{nonzero pattern} $\cT$ is an $n\times n$ matrix with entries in $\{0,*\}$ 
and $Q(\cT)$ is the set of real matrices $A$ with $A_{ij}\neq 0$ if and only if $\cT_{ij}\neq 0$.
Matrices in $Q(\A)$ are called \emph{realizations} of $\A$.
A \emph{superpattern} of a sign (or nonzero) pattern $\cT$ is a pattern $\A'$ such that $\A'_{ij}=\cT_{ij}$ whenever $\cT_{ij}\neq 0$.
A pattern $\cT$ is \emph{sign-nonsingular} if $\det(A)\neq 0$ for every $A\in Q(\cT).$ The symbols $\ostar$, 
$\omin$, 
and $\oplu$ represent an arbitrary real number, a nonpositive real number, and a nonnegative real number, respectively, in a matrix.

A pattern $\cT$ \emph{allows} a particular matrix property $\ccP$, if there exists  a real matrix $A\in Q(\cT)$ such that $A$ has the property $\ccP$. Likewise, a pattern
$\A$ \emph{requires} a property $\ccP$ if every matrix $A\in Q(\A)$ has property $\ccP$. 
 By definition, every pattern that allows (respectively requires) the nSSP also allows (resp. requires) the nSMP.

\begin{example}\label{ex:first} Let 
$$A=\left[ \begin{array}{rr}
1&0\\
-3&2\end{array}
\right] \qquad
~~\mbox{and}~~
\cT=\left[ \begin{array}{cc}
+&0\\
-&+\end{array}
\right].
$$
Then $A\in Q(\cT).$
Further, one can check that $A$ has the nSMP: if $A\circ X=\O$, 
then $X_{11}=X_{21}=X_{22}=0$.  But then
$(X^TA-AX^T)_{21}=X_{12}$,  and thus $X_{12}=0$ if $X^TA-AX^T=\O$. 
Therefore $X=\O$ will be the only solution to the nSMP equations. In fact, $A$ has the nSSP.

For a second example in $Q(\A)$, let 
$$B=\left[ \begin{array}{rr}
1&0\\
-3&1\end{array}
\right] 
\qquad ~~\mbox{and}~~
X=\left[ \begin{array}{cc}
0&4\\
0&0\end{array}
\right]. \qquad 
$$
Then $B\circ X=\O$,
$X^TB-BX^T=\O$, and
${\rm tr}(X^T)={\rm tr}(X^TB)={\rm tr}(X^TB^2)=0.$ Thus
$B$ does not have the nSMP (and hence also does not have the nSSP). 
We conclude that $\cT$ does not require the nSMP, even though $\cT$ does allow the nSMP. 

Taking the pattern $\cT$ in Example~\ref{ex:first} and replacing the zero entry with a signed entry gives a proper superpattern of $\cT$ that
requires the nSMP since the Hadamard product condition (\ref{eq:had}) forces the trivial solution
in this case.
\end{example}

A matrix $B$ is \emph{permutationally similar} to $A$ if
there is a permutation matrix $P$ such that $B=P^{-1}AP.$
A matrix $B$ is \emph{diagonally similar} to $A$ if there is a diagonal matrix
$D$ such that $B = D^{-1}AD$.
A pattern $\A$ is \emph{signature similar} to $\B$ if there is a diagonal matrix $D$ with entries $\pm 1$ on the diagonal such that 
$DAD\in Q(\B)$ for every $A\in Q(\A)$. Note that $D^{-1}=D$ in this case.  

\begin{lemma}\cite{CGSV}\label{lem:equiv}
If $B$ is a matrix obtained from $A$ via permutation similarity, diagonal similarity, 
nonzero scalar multiplication and/or transposition, then $B$ will have the nSMP if $A$ has the nSMP.
\end{lemma}

Given Lemma~\ref{lem:equiv}, we say that two sign patterns are \emph{equivalent} if one can be obtained from the other via permutation similarity, signature similarity, transposition and/or negation (a scalar multiplication).

The digraph $D(\A)$ of an $n\times n$ matrix pattern $\A$ is obtained by viewing the matrix as an adjacency matrix for a digraph
on $n$ vertices $v_1,\ldots, v_n$ such that there is a signed arc
from $v_i$ to $v_j$ if and only if $\A_{ij}\neq 0$, with the sign of the arc being $\A_{ij}$. Based on Lemma~\ref{lem:equiv}, if two digraphs are isomorphic, then the corresponding sign patterns are equivalent. In addition, reversing all the arcs of a digraph will also give an equivalent sign pattern.

A directed (or simple) \emph{$k$-cycle} is a digraph on vertices 
$v_1,\ldots, v_k$ with the set of arcs
$$\{(v_1,v_2), (v_2,v_3), \ldots, (v_{k-1},v_k),  (v_k,v_1)\}.$$ The \emph{sign} of a cycle is the product of the signs of the arcs of cycle. A \emph{composite $k$-cycle} (or composite cycle of length $k$) in a digraph is a collection of disjoint cycles on exactly $k$ vertices.
A $1$-cycle is often referred to as a \emph{loop} in the digraph and corresponds to a nonzero diagonal entry in the corresponding matrix. 

A matrix $A$ is \emph{reducible} if it is permutationally similar to a block triangular matrix, that is, if
\begin{equation}\label{eq:red}
P^TAP=\left[
\begin{array}{c|c}
A_1&B \\ \hline
\O&A_2
\end{array}\right]
\end{equation}
for some permutation matrix $P$ and square matrices
$A_1$ and $A_2$ (called \emph{diagonal blocks} of $A$). 
If a matrix is not reducible, then it is \emph{irreducible}. Note that the eigenvalues of $A$ are simply the eigenvalues of $A_1$ and $A_2$
if $A$ satisfies (\ref{eq:red}).
It is a known fact (see e.g. \cite{HJ}) that $A$ is irreducible if and only if the digraph of $A$ is strongly connected, that is, there is a directed path between
any pair of vertices in the graph.

\begin{theorem}
\cite{CGSV}\label{thm:reducible}
Suppose $A$ is reducible as in (\ref{eq:red}). 
Then $A$ has the nSMP if and only if 
$A_1$ and $A_2$ have no eigenvalues in common and
both $A_1$ and $A_2$ have the nSMP.
\end{theorem}

Let $q(A)$ denote the number of distinct eigenvalues of the matrix $A$.
The following result is direct consequence of the bifurcation theorem for nSMP in \cite{CGSV}. 

\begin{theorem}\cite{CGSV}\label{thm:bifurcation}
If $A\in Q(\A)$ is an $n\times n$ matrix that has the nSMP,
then $\A$ allows 
$n$ distinct eigenvalues. If, in addition, $A$ has a repeated real eigenvalue, then
$\A$ allows 
$k$ distinct eigenvalues for every $k$ with $q(A)\leq k\leq n$. 
\end{theorem}

Note that the conclusion of the second statement of Theorem~\ref{thm:bifurcation} was incorrectly given in \cite{CGSV} without the hypothesis of a repeated real eigenvalue.

\section{Patterns that require the nSMP.}

As already noted, any pattern that requires the nSSP will also require the nSMP. In \cite{Li}, the \emph{similarity-transversality property} (STP) was introduced and shown to be equivalent to the nSSP. As such, some patterns in \cite{Li} were demonstrated to require the nSSP, such as any full tridiagonal pattern, any full Hessenberg pattern, or any full backward upper triangular pattern. In this section we develop other patterns that require the nSMP.

\subsection{Patterns with all off-diagonal entries nonzero}\label{sec:hollow}
In this subsection, we note that if $\A$ is a sign pattern for which $\mathcal{D}(\A)$
is a directed complete graph (which  may or may not have some loops),
then $\A$ requires the nSMP.

\begin{lemma}\label{lem:partial}
Suppose $\A$ is an $n\times n$ matrix with $n\geq 2$ such that for some $i\neq j$, row  and column $i$ of $\A$, as well as row and column $j$ of $\A$ have exactly one zero entry, and  $\A_{ii}=\A_{jj}=0$.   
Suppose $A\in Q(\A)$.  If $A\circ X=\O$ and
$X^TA-AX^T=\O$,
then $X_{ii}=X_{jj}$
\end{lemma}

\begin{proof}
This follows from the fact that $(AX^T-X^TA)_{ij}=A_{ij}(X_{ii}-X_{jj})$.
\end{proof}

\begin{theorem}\label{thm:hollow} If all the off-diagonal entries of a sign pattern $\A$ are nonzero,  then $\A$ requires the nSMP.
\end{theorem}

\begin{proof}
Suppose all the off-diagonal entries of $\A$ are nonzero.
If $\A$ has no zero entries on the diagonal, then 
$A\circ X=\O$ implies $X=\O$. If 
$\A$ has  exactly one zero on the diagonal, then $A\circ X=\O$ demonstrates that all entries of $X$ are zero except possibly one diagonal entry. In this case, ${\rm{tr}}(X^T)=0$ implies that $X=\O$.
    If $\A$ has more than one zero entry on the diagonal, then Lemma~\ref{lem:partial} implies any nonzero diagonal entries of $X$ must be equal. 
    And $A\circ X=\O$ implies $X$ has no nonzero entries off the diagonal. Thus ${\rm{tr}}(X^T)=0$ implies $X=\O$.
    Therefore $\A$ requires the nSMP.
\end{proof}

\subsection{Cycle patterns}\label{sec:cycle}

In \cite{CGSV} it was shown that if the digraph of an $n\times n$ sign pattern $\A$ is a directed $n$-cycle with no loops, 
then $\A$ requires the nSMP. Below we extend this result, by showing that a cycle can have any number of loops
and the pattern will still require the nSMP.

\begin{theorem}\label{thm:cycle}
If the digraph of an $n\times n$ sign pattern $\A$ is an $n$-cycle, with any number of vertices having a loop, then $\A$ requires the nSMP.
\end{theorem}

\begin{proof} Suppose $D(\A)$ is an $n$-cycle.
Suppose $A\in Q(\A)$ and $A\circ X=\O$, $AX^T=X^TA$, and ${\rm{tr}}(X^TA^k)=0$ for $0\leq k<n$. Note that ${\rm{tr}}(X^TA^n)=0$ also holds since $A^n$ is a polynomial in $A$ of degree $n-1$ by the Cayley-Hamilton Theorem. 
By permutation and diagonal similarity,
we may assume that 
\[ A=\left[ \begin{array}{ccccc}
\ostar & 1 &0 &\cdots &0\\
0&\ddots&\ddots&\ddots&\vdots\\
\vdots&\ddots&\ddots&\ddots&0\\
0& \cdots &0&\ostar&1\\
d& 0&\cdots&0&\ostar\\
\end{array}\right]
\]
for some $d\in \R$, with the diagonal entries undetermined,
as denoted by $\ostar$. In our multiplications, we will use
the fact that $X_{ii}A_{ii}=0$ for $1\leq i\leq n$, which follows from $A\circ X=\O$. 
Let 
\[ P=\left[\begin{array}{ccccc}
0&\cdots&\cdots &0&1\\
1&0&\cdots&0&0\\
0&\ddots&\ddots&\vdots&\vdots\\
\vdots&\vdots&\ddots&0&0\\
0&\cdots&0&1&0\\
\end{array}\right].
\]
We will show by induction that $P^k\circ X^T=\O$ for $1\leq k\leq n.$
Note that $P\circ X^T=\O$ since $A\circ X=\O$. 
Letting $\mathbf{1}$ be the $n\times 1$ all-ones vector, 
note that $AX^T=X^TA$ implies $[P\circ (AX^T)]\one=[P\circ (X^TA)]\one$ and hence
\[ \left[X^T_{2,n},X^T_{3,1},X^T_{4,2},\ldots,X^T_{n,n-2}, dX^T_{1,n-1} \right] =\left[ X^T_{1,n-1},dX^T_{2,n}, X^T_{3,1},\ldots, X^T_{n-1,n-3},X^T_{n,n-2} \right]. \]
Therefore $X^T_{1,n-1}=X^T_{2,n}=u_1$ for some $u_1\in \R$ and $X^T_{3,1}=X^T_{4,2}=\cdots =X^T_{n-1,n-3}=X^T_{n,n-2}=du_1$. 
Now, \[{\rm{tr}}(X^TA^2)=dX^T_{1,n-1}+dX^T_{2,n}+ X^T_{3,1}+\cdots+X^T_{n,n-2}=ndu_1.\] 
Thus ${\rm{tr}}(X^TA^2)=0$ implies $u_1=0$
and thus $X^T_{1,n-1}=X^T_{2,n} = \cdots = X^T_{n,n-2}=0$.
Consequently $P^2\circ X^T=\O$.
Inductively, using the fact that  $AX^T=X^TA$ implies $[P^k\circ (AX^T)]\one=[P^k\circ (X^TA)]\one$, it follows that
\[ \left[X^T_{2,n-k+1},X^T_{3,n-k+2},\ldots, X^T_{n,n-k-1}, dX^T_{1,n-k}\right] 
= 
\]
\[ 
\left[ X^T_{1,n-k},X^T_{2,n-k+1},\ldots,X^T_{k,n-1},dX^T_{k+1,n},X^T_{k+2,1},\ldots, X^T_{n,n-k-1} \right]. \]
 Thus $X^T_{1,n-k}=X^T_{2,n-k+1}=\cdots =X^T_{k+1,n}=u_k$ for some $u_k\in \R,$
 and $X^T_{k+2,1}=X^T_{k+3,2}=\cdots=X^T_{n,n-k-1}=du_k$. Now, 
\[0={\rm{tr}}(X^TA^{k+1})=d\sum_{i=0}^k X^T_{1+i,n-k+i} + \sum_{i=1}^{n-k-1} X^T_{k+1+i,i}=ndu_k.\] Thus $u_k=0$ and $P^{k+1}\circ X^T=\O$ for $0\leq k\leq n-1$. 
Therefore $X=\O$ and so $A$ has the nSMP. Therefore $\A$ requires the nSMP.
\end{proof}

\subsection{Star patterns}\label{sec:star}
In this subsection, we characterize when star sign patterns require, allow, or do not allow the nSMP.
Let \[ \A=
\left[
\begin{array}{ccccc}
    \ostar &*&\cdots&\cdots&* \\
    \sstar & \ostar &0 &\cdots&0\\  
    \vdots &0 & \ddots&\ddots &\vdots\\
    \sstar &\vdots &\ddots&\ostar&0\\
    \sstar &0&\cdots&0&\ostar\\
\end{array}
\right]
\]
be a \emph{nonzero star pattern}. 
A \emph{star sign pattern} has the structure of the nonzero star 
pattern with each $*$ replaced by an element in $\{+,-\}$
and each $\ostar$ replaced by an element of $\{ 0,+,-\}$. The \emph{centre} of the star is vertex $v_1$ in $D(\A)$.

\begin{example}\label{ex:starloops}
Suppose $\A$ is a star sign pattern having two non-centre vertices without loops and suppose
the $2$-cycles incident to the loops have different signs. By permutation and signature
similarity we may assume $\A_{n-1,n-1}=\A_{n,n}=0$, $\A_{1,n-1}=\A_{1,n}=\A_{n,1}=+$ and $\A_{n-1,1}=-$. 
Let $A\in Q(\A)$ with 
$A_{1,n-1}=A_{1,n}=A_{n-1,1}=1$ 
and $A_{n,1}=-1$.
Let $S=\{n-1,n\}$ and  $X$ be an $n\times n$ matrix with principal submatrix 
\[ 
X[S]=\left[
\begin{array}{rr}
-1&1\\-1&1
\end{array}\right]
\]
with all other entries
of $X$ zero. Then $A\circ X=\O$, $X^TA=\O$ and $AX^T=\O$ so $X^TA-AX^T=\O$. Further ${\rm{tr}}(X)=0$
and $X^TA=\O$ implies ${\rm{tr}}(X^TA^k)=0$ for $1\leq k\leq n-1$. Therefore $A$ does not
have the nSMP and so $\A$ does not require the nSMP.
\end{example}

\begin{example}\label{ex:starthree}
Suppose $\A$ is a star sign pattern having three non-centre vertices having a loop of the same sign.
By permutation similarity, we may assume that vertices $n-2,n-1$ and $n$ all have a loop
of the same sign in $D(\A)$. By negation, we can assume the signs of these loops are positive. By signature similarity we may assume
that $\A_{1,n-2}=\A_{1,n-1}=\A_{1,n}=+$.
Let $A\in Q(\A)$ be a matrix  such that 
$A_{1,n-2}=A_{1,n-1}=A_{1,n}=1$, and $A_{n-2,n-2}=A_{n-1,n-1}=A_{n,n}=1$ with
$A_{n-2,1}=a$, $A_{n-1,1}=b$ and $A_{n,1}=c$ for some nonzero $a,b,c\in \R$.  
Taking $S=\{n-2,n-1,n\}$, let $X$ by an $n\times n$ matrix with principal submatrix 
\[
X[S]= \left[ 
\begin{array}{rrr}
0&cb&-cb\\
-ca&0&ca\\
ab&-ab&0\\
\end{array}
\right]
\]
with the remaining entries of $X$ zero. 
Then $A\circ X=\O$, and 
$X^TA=X^T=AX^T$. Since ${\rm{tr}}(X^T)=0$, and 
$X^TA=X^T$, it follows that ${\rm{tr}}(X^TA^k)=0$ for
$0\leq k<n.$
So $A$ does not have the nSMP
and hence $\A$ does not require
the nSMP. If we replace the condition that  three non-centre vertices have the same sign,
with the condition that there are three non-centre vertices with no loops,  then the same
argument above still applies (noting that
$X^TA=\O=AX^T$). Therefore $\A$ also does not require the nSMP if $\A$ is a star sign pattern with three loopless non-centre vertices.
\end{example}

\begin{theorem}\label{thm:star}
Suppose $\A$ is an $n\times n$  star sign pattern for some $n\geq 2$. Then $\A$ requires the nSMP if and only if 
 $D(\A)$ does not have more than two non-centre vertices with loops of the same sign (including possibly a ``sign'' of zero) 
and in addition, if $D(\A)$ has two non-centre vertices with no loop, then the incident $2$-cycles have the same sign.
\end{theorem}

\begin{proof}
Examples~\ref{ex:starloops} and \ref{ex:starthree} demonstrate the two given conditions are necessary.
Suppose $\A$ is a star sign pattern such that $D(\A)$ has no more than two non-centre
vertices having loops with the same sign (including the ``sign'' of zero). 
Note that if $n=2$, then $\A$ requires the nSMP by Theorem~\ref{thm:hollow}. Thus assume $n\geq 3$.
Let $A\in Q(\A)$. By Lemma~\ref{lem:equiv}, using diagonal similarity, we may assume that
$A_{1,j}=1$ for $2\leq j\leq n$. Suppose $A\circ X=\O$, $X^TA-AX^T=\O$ and ${\rm{tr}}(X^TA^k)=0$ for 
$0\leq k<n$. We will demonstrate that this implies $X=0$ and hence $\A$ requires the nSMP. 

The first row and column of $X$ will be zero except possibly $X_{11}$
since $A\circ X=\O$. Note that for $i,j \in \{2,\ldots,n\}$, $(X^TA-AX^T)_{ij}=(X^T)_{ij}(A_{jj}-A_{ii})$. It follows
that for any $i,j \in \{2,\ldots,n\}$ with $A_{ii}\neq A_{jj}$, then $X_{ij}=X_{ji}=0$.    
Suppose $A_{ii}=A_{jj}\neq 0$.  Then column $i$ of $X^T$ is $X_{ji}^Te_j$ 
where $e_j$ is the standard
basis vector with a 1 in the $j$th position (and column $j$ is $X_{ij}^Te_i$). Likewise
row $i$ of $X^T$ is $X_{ij}^Te_j^T$ and row $j$ is $X_{ji}^Te_i^T$. We take these observations into
account in the following analysis.

Suppose that every non-centre vertex has a loop. Since $A\circ X=\O$, $X$ has at most one
nonzero diagonal entry $X_{11}$. However, the trace condition will force $X_{11}=0$ in this
case. Then for any $2\leq j\leq n$, 
$(X^TA-AX^T)_{1j}=-(AX^T)_{1j}$ which forces $X=\O$ by the above observations.   

Suppose that for some diagonal entry corresponding to a non-centre vertex, there is no other
diagonal entry corresponding to a non-center vertex with the same sign in $\{0,+,-\}$. Without loss
of generality, this entry is $A_{nn}$. Then row and column $n$ of $X$ is all zero by the previous
paragraph. Further, $(X^TA-AX^T)_{1n}=X_{11}$, hence $X_{11}=0$.  Suppose $A_{ii}=A_{jj}\neq 0$ for some $i,j$. Note that
$(X^TA-AX^T)_{1i}=-X^T_{ji}$ and $(X^TA-AX^T)_{1j}=-X^T_{ij}$. Thus rows (and columns) $i$ and $j$ of $X^T$
are all zero. Suppose there exists two non-centre vertices with no loops. Without loss of generality,
assume $A_{22}=0$ and $A_{33}=0$. Note that $X^T_{33}=-X^T_{22}$ since ${\rm{trace}}(X)=0$. 
Further $(X^TA-AX^T)_{12}=-X^T_{32}-X^T_{22}$ and $(X^TA-AX^T)_{13}=-X^T_{23}-X^T_{33}$. Therefore
$X^T_{23}=X^T_{22}$ and $X^T_{32}=-X^T_{22}$. Now $(X^TA-AX^T)_{21}=(A_{21}+A_{31})X^T_{22}$. However, 
$A_{21}$ has the same sign as $A_{31}$ by the hypothesis that the $2$-cycles incident to non-centre vertices have the same sign.  
Thus $X^T_{22}=0$ and therefore $X=\O$. 

Suppose now that any entry in $\{0,+,-\}$ that appears on the diagonal at a non-centre vertex of $D(\A)$ 
appears exactly twice. As a result, we need only consider the following  three cases that do not have loops on all non-centre vertices.

Case (1). Suppose 
\[
A=\left[ \begin{array}{ccc}
m&1&1\\
a&0&0\\ b&0 &0
\end{array}\right]
\qquad {\rm{ with}} \qquad
X^T=\left[ \begin{array}{ccc}
x&0&0\\
0&r&s\\ 0&t&u
\end{array}\right]
\]
for some nonzero $a,b$ with the same sign. Suppose $m\neq 0$. Then $x=0$. 
Further, $(X^TA-AX^T)_{12}=-r-t$, so that $t=-r$. Similarly $u=-s$.
Since ${\rm{trace}}(X)=0$, $s=r$. Now $(X^TA-AX^T)_{21}=(a+b)r$. Since
$a$ and $b$ have the same sign, $r=0$. Therefore $X=\O$. Suppose instead
that $m=0$. Then $(X^TA-AX^T)_{12}=x-r-t$ and $(X^TA-AX^T)_{13}=x-s-u$
so that $t=x-r$ and $u=x-s$. Then ${\rm{trace}}(X^TA^2)=2(a+b)x$. Since
$a,b$ have the same sign, $x=0$. Since ${\rm{trace}}(X^T)=0$, 
$s=r$. In this case $(X^TA-AX^T)_{21}=(a+b)r$. Thus $r=0$, since $a,b$ have the
same sign, and therefore $X=\O$.

Case (2). Suppose 
\[
A=\left[ \begin{array}{ccccc}

m&1&1&1&1\\
a&0&0&0&0\\ b&0&0&0&0\\ 
c&0&0&g&0\\ d&0&0&0&h\\
\end{array}\right]
\qquad {\rm{ with}} \qquad
X^T=\left[ \begin{array}{ccccc}
x&0&0&0&0\\
0&r&s&0&0\\
0&t&u&0&0\\
0&0&0&0&v\\ 
0&0&0&w&0
\end{array}\right]
\]
for some nonzero $a,b$ with the same sign and some nonzero $g,h$ with the same sign.

Suppose $m\neq 0$. Then $x=0$ since $A\circ X=\O$. Now $(X^TA-AX^T)_{12}=-r-t$ and $(X^TA-AX^T)_{13}=-s-u$.
Thus $t=-r$ and $u=-s$. Since ${\rm{trace}}(X)=0$, $s=r$. Now $(X^TA-AX^T)_{21}=(a+b)r$. Since
$a$ and $b$ have the same sign, $r=0$. Now $(X^TA-AX^T)_{14}=-w$ and $(X^TA-AX^T)_{15}=-v$, and 
thus $w=v=0$. Therefore $X=0$.

Suppose $m=0$. Then $(X^TA-AX^T)_{12}=x-r-t$ and $(X^TA-AX^T)_{13}=x-s-u$
so that $t=x-r$ and $u=x-s$. Now $(X^TA-AX^T)_{14}=x-w$ and $(X^TA-AX^T)_{15}=x-v$,
thus $v=w=x$. If $x=0$, then $X=\O$ as per
the argument in the previous paragraph when $m\neq 0$.   
Thus suppose $x\neq 0$.
Since $(X^TA-AX^T)_{51}=(c-d)x$ 
and $(X^TA-AX^T)_{54}=(g-h)x$, $c=d$ and $g=h$.
But then $0={\rm{trace}}(X^TA^3)=6cgx$ contradicting the fact that $c,g$ and $x$ are nonzero.
Therefore $X=\O$.

Case (3). Suppose 
\[
A=\left[ \begin{array}{rrrrrrr}
m&1&1&1&1&1&1\\
a&0&0&0&0&0&0\\ b&0&0&0&0&0&0\\ 
c&0&0&g&0&0&0\\ d&0&0&0&h&0&0\\ e&0&0&0&0&-k&0\\ f&0&0&0&0&0&-\ell \end{array}\right]
\qquad {\rm{ with}} \qquad
X^T=\left[ \begin{array}{ccccccc}
x&0&0&0&0&0&0\\ 0&r&s&0&0&0&0\\
0&t&u&0&0&0&0\\ 0&0&0&0&v&0&0\\ 
0&0&0&w&0&0&0\\ 0&0&0&0&0&0&y\\ 0&0&0&0&0&z&0
\end{array}\right]
\]
for some nonzero $a,b$ with the same sign, and some positive $g,h,k$ and $\ell$. By Lemma~\ref{lem:equiv}, using scaling (and diagonal similarity) we 
can assume $\ell=1$.

Suppose $m\neq 0$. Then $x=0$ since $A\circ X=\O$. Now $(X^TA-AX^T)_{12}=-r-t$ and $(X^TA-AX^T)_{13}=-s-u$, and 
thus $t=-r$ and $u=-s$. Since ${\rm{trace}}(X)=0$, $s=r$.
Now $(X^TA-AX^T)_{21}=(a+b)r$. Since
$a$ and $b$ have the same sign, $r=0$.
Now $(X^TA-AX^T)_{14}=-w$ and $(X^TA-AX^T)_{15}=-v$, 
and $(X^TA-AX^T)_{16}=-z$ and $(X^TA-AX^T)_{17}=-y$.
Thus $w=v=z=y=0$. Therefore $X=0$.

Suppose $m=0$.  Then
$(X^TA-AX^T)_{12}=x-r-t$ and $(X^TA-AX^T)_{13}=x-s-u$,
so that $t=x-r$ and $u=x-s.$ 
Now $(X^TA-AX^T)_{14}=x-w$, $(X^TA-AX^T)_{15}=x-v$,
  $(X^TA-AX^T)_{16}=x-z$ and $(X^TA-AX^T)_{17}=x-y$,
thus $v=w=z=y=x$. If $x=0$, then $X=\O$, as per the argument in the previous paragraph when $m\neq 0$.
Suppose $x\neq 0$. We claim that this leads to a contradiction.
Since $(X^TA-AX^T)_{51}=(c-d)x$ and $(X^TA-AX^T)_{54}=(g-h)x$, 
 then $d=c$ and $h=g$.
Since $(X^TA-AX^T)_{71}=(e-f)x$ and $(X^TA-AX^T)_{76}=(1-k)x$, $e=f$ and $k=1$.
Now $s=r+2x$ since ${\rm{trace}}(X^T)=r-s+2x$. 
At this point,
\[
A=\left[ \begin{array}{rrrrrrr}
0&1&1&1&1&1&1\\
a&0&0&0&0&0&0\\ b&0&0&0&0&0&0\\ 
c&0&0&g&0&0&0\\ c&0&0&0&g&0&0\\ f&0&0&0&0&-1&0\\ f&0&0&0&0&0&-1 \end{array}\right]
\qquad {\rm{ with}} \qquad
X^T=\left[ \begin{array}{ccccccc}
x&0&0&0&0&0&0\\ 0&r&r+2x&0&0&0&0\\
0&x-r&-x-r&0&0&0&0\\ 0&0&0&0&x&0&0\\ 
0&0&0&x&0&0&0\\ 0&0&0&0&0&0&x\\ 0&0&0&0&0&x&0
\end{array}\right].
\]
 Then $f=cg$ since ${\rm{trace}}(X^TA^3)=6x(cg-f)$ and $x\neq 0$.
Then ${\rm{trace}}(X^TA^5)=10cgx(g^2-1)$. Since $c,g$ and $x$ are nonzero,
and $g$ is positive, $g=1$. Now ${\rm{trace}}(X^TA^2)=2x(a+b) + 8cx$. Thus
$c=\frac{a+b}{4}$. Then ${\rm{trace}}(X^TA^4)=-4x(a+b)$. 
However, $a,b$ have the
same sign and $x$ is nonzero, which is a contradiction.

Therefore $\A$ requires the nSMP. 
\end{proof}

\begin{corollary}
If $\A$ is an $n\times n$ star sign pattern and $n>7$, then $\A$ does not require the nSMP.
\end{corollary}

In Section~\ref{sec:hessy}, 
the star sign patterns that do not allow the nSMP will be characterized (Corollary~\ref{cor:starallow}).

\begin{remark}
    If a zero-nonzero matrix pattern requires the nSMP, then any sign pattern with the same zero-nonzero pattern will also require the nSMP. However, Theorem~\ref{thm:star} demonstrates that there is a sign pattern that requires the nSMP, but the corresponding zero-nonzero pattern does not require the nSMP  since the characterization depends on the signs.
\end{remark}

\subsection{Hessenberg patterns}\label{sec:hessenberg}

In this subsection we characterize a large class of Hessenberg patterns that require the nSMP.
A proper (lower) \emph{Hessenberg} pattern is an $n\times n$ matrix pattern $\A$ such that for $1\leq i < n$, 
$\A_{i,i+1}\neq 0$, and $\A_{i,j}=0$ for all $j\geq i+2$.  
For $k\in\{1,2,\ldots, n-1\}$, the \emph{$k$th subdiagonal} of an $n \times n$ matrix is the set of positions $\{ (i, i-k) : k+1 \leq i \leq n\}$. Note that if the digraph of a (lower) Hessenberg pattern $\A$ has a $k$-cycle, then the matrix $A\in Q(\A)$ has a nonzero entry on the 
$(k-1)$th subdiagonal of $A$. 

\begin{example} The following two Hessenberg patterns have exactly one nonzero entry on each subdiagonal. Theorem~\ref{thm:subdiag} will demonstrate that these patterns require the nSMP. In fact, any superpattern of these patterns, which preserves the Hessenberg structure, will also require the nSMP by Theorem~\ref{thm:subdiag}.
\[
\left[ 
\begin{array}{cccccc}
0&+&0&0&0&0\\
+&0&+&0&0&0\\
+&0&0&+&0&0\\
+&0&0&0&+&0\\
+&0&0&0&0&+\\
+&0&0&0&0&0\\
\end{array}
\right] \qquad 
\left[ 
\begin{array}{cccccc}
0&+&0&0&0&0\\
0&0&+&0&0&0\\
0&-&0&+&0&0\\
+&0&0&0&+&0\\
0&0&-&0&0&+\\
+&-&0&0&0&0\\
\end{array}
\right]
\]
\end{example}

Proper Hessenberg matrices are known to be
nonderogatory (see, e.g. \cite{HJ}): a matrix is \emph{nonderogatory} if the geometric multiplicity of each of its eigenvalues is one.  The following theorem from \cite[Theorem 3.2.4.2]{HJ} provides a useful
tool for demonstrating that some nonderogatory matrices have the nSMP. 

\begin{theorem}\cite{HJ}\label{rem:Hessy} If a matrix $X$ commutes with a nonderogatory matrix $A$
then 
\[X=c_0I+c_1A+c_2A^2+\cdots +c_{n-1}A^{n-1}
\]
for some $c_0,c_1,\ldots, c_{n-1}\in \R$.
\end{theorem}

\begin{theorem}\label{thm:subdiag}
    If $\A$ is a proper Hessenberg pattern whose digraph has at least one $k$-cycle for each $k$, $2\leq k\leq n$, then $\A$ requires the nSMP.
\end{theorem}

\begin{proof}
    Let $\A$ be a proper Hessenberg pattern 
    with at least one nonzero entry on each subdiagonal.
    Let $A\in Q(\A).$ Suppose $A\circ X=\O$, $AX^T-X^TA=\O$, and ${\rm{tr}}(X^TA^k)=0$ for $0\leq k<n.$
    Since $A$ is a proper Hessenberg matrix, $A$ is nonderogatory. By Theorem~\ref{rem:Hessy},
    $X^T=c_0I+c_1A+\cdots +c_{n-1}A^{n-1}$ for
    some $c_0,c_1,\ldots,c_n\in \R$. Since $(X^T)_{1n}=0$ and $(A^k)_{1n}\neq 0$ only for $k=n-1$,
    it follows that $c_{n-1}=0.$ Iterating, we can demonstrate that $c_{n-2}=c_{n-3}=\cdots =c_1=0.$ For suppose $X^T=c_0I+c_1A+\cdots + c_jA^j$ for some $j$, $1\leq j<n$. Then  $(A^j)^T$ has all nonzero entries on the $j$th subdiagonal, but the smaller powers of $A^T$ have only zero entries on the $j$th subdiagonal. However, $X$ has a zero entry on the $j$th subdiagonal by hypothesis. Thus $c_j=0$. Hence $X^T=c_0I.$ But then ${\rm{tr}}(X^T)=0$ implies $c_0=0$.
    Therefore $X=\O$, and thus $A$ has the nSMP. Therefore $\A$ requires the nSMP.
\end{proof}

\section{Patterns that allow or require distinct eigenvalues}\label{sec:hessy}

In this section we characterize the patterns $\A$ that allow the nSMP, first in terms of patterns that allow distinct eigenvalues, then in terms of the number of vertices in a largest composite cycle in $D(\A)$. This allows a characterization of the star patterns that do not allow the nSMP. Further, an open question of \cite{CGSV} is answered in the affirmative.

\begin{theorem}\label{thm:hessy}
    An $n\times n$ sign pattern $\A$ allows the nSMP if and only if $\A$ allows distinct eigenvalues.
\end{theorem}

\begin{proof}
If $\A$ allows the nSMP, then there is a matrix $A\in Q(\A)$ that has the nSMP. By Theorem~\ref{thm:bifurcation},
there is a matrix $B\in Q(\A)$ with distinct eigenvalues. Thus $\A$ allows distinct eigenvalues. 

For the converse,  suppose $\A$ is a pattern
that allows distinct eigenvalues. Then there is a matrix $A\in Q(\A)$ with 
eigenvalues $\lambda_1,\lambda_2,\ldots, \lambda_n$ which are all distinct. Note that this implies that $A$ is nonderogatory. Let $V$ be an $n\times n$ matrix with $V_{ij}=\lambda_i^{j-1}$.
Note that $V$ is a Vandermonde matrix, and hence is a nonsingular matrix since the eigenvalues of $A$ are distinct.
Suppose $A\circ X=\O$, $X^TA=AX^T$, and ${\rm{tr}}(X^TA^k)=0$ for $0\leq k\leq n-1.$ 
By Theorem~\ref{rem:Hessy},
$X^T=\displaystyle\sum_{k=0}^{n-1} c_kA^k$
for some $\mathbf{c}=(c_0,c_1,\ldots,c_{n-1})\in \R^n$. 
Note that, for $0\leq k\leq n-1$, 
\[0={\rm{tr}}(X^TA^k)={\rm{tr}}\left(\displaystyle\sum_{i=1}^{n} c_{i-1}A^{k+i-1}\right)=
\displaystyle\sum_{i=1}^{n} c_{i-1}{\rm{tr}}(A^{k+i-1}).
\] 
Since the trace of a matrix is the same as the sum of its eigenvalues, for $1\leq j\leq n,$
\[0=\displaystyle\sum_{i=1}^{n} c_{i-1}{\rm{tr}}(A^{j+i-2})=
\displaystyle\sum_{i=1}^{n} c_{i-1}  \left(  \displaystyle\sum_{m=1}^n \lambda_m^{j+i-2}\right)=
\displaystyle\sum_{i=1}^{n} c_{i-1}  (V^TV)_{ij}.\]
This system can be rewritten as $\mathbf{c}(V^TV)=\mathbf{0}.$  Since $V$ is nonsingular, $\mathbf{c}=\mathbf{0}$ and hence $X^T=\O$.
Therefore $A$ has the nSMP and $\A$ allows the nSMP.
\end{proof}

The $n\times n$ sign patterns that allow distinct eigenvalues are known to be those whose digraph includes a composite cycle of length  
$n-1$ or $n$ (see, e.g. \cite[Lemma (p.172)]{EJ}).

\begin{corollary}\label{cor:n-1}
An $n\times n$ sign pattern $\A$ allows the nSMP if and only if the largest composite cycle length in $D(\A)$ is at least $(n-1)$.
\end{corollary}

\begin{corollary}\label{cor:starallow}
An $n\times n$ star sign pattern  $\A$ does not allow the nSMP if and only if 
$D(\A)$ has less than $(n-3)$ non-centre loops.
\end{corollary}

A question was raised in \cite{CGSV}, whether every pattern that requires distinct eigenvalues must also require the nSMP.
The proof of Theorem~\ref{thm:hessy} 
demonstrates that if $A$ has distinct eigenvalues, then $A$ has the nSMP,
answering the question in the affirmative: 

\begin{theorem}\label{thm:requires}
If a sign pattern 
$\A$ requires distinct eigenvalues, then $\A$ requires the nSMP.    
\end{theorem}

The next example shows that the converse of Theorem~\ref{thm:requires} does not hold.

\begin{example} Let 
\[  \A=\begin{bmatrix}
0&+&0\\
+&0&+\\
+&0&0\\
\end{bmatrix} {\rm {\qquad and \qquad }}
A=\begin{bmatrix}
0&1&0\\
3&0&1\\
2&0&0\\
\end{bmatrix}.
\]
Then $\A$ requires the nSMP, by Theorem~\ref{thm:subdiag}. 
However, $\A$ does not require distinct eigenvalues since 
$A\in Q(\A)$ and $A$ has a repeated eigenvalue of $-1$. (That $\A$ does not require distinct eigenvalues can also be determined from the characterization of $3\times 3$ patterns that require distinct 
eigenvalues in \cite{LH}.)
\end{example}

\section{Patterns that allow but do not require the nSMP}\label{sec:allows}

In this section, we describe the construction of some patterns that allow the nSMP but do not require the nSMP. We start by stating a general result about bipartite patterns, then describe some patterns obtained by appending a star pattern to another pattern based on examples in Section~\ref{sec:star}, and finally make some observations about spectrally arbitrary patterns.

A sign pattern $\A$ is \emph{bipartite} if the underlying graph 
of $D(\A)$ is bipartite. Thus a bipartite pattern has no cycles of odd length 
 in $D(\A)$.

\begin{theorem}\label{thm:anotr}
Let $\A$ be an $n\times n$ bipartite sign pattern having a pair of intersecting oppositely signed cycles of different lengths 
and suppose that either 
$n$ is even and $\A$ is sign-nonsingular, or
the largest composite cycle of $D(\A)$ is 
length $(n-1)$. Then $\A$ allows 
the nSMP but does not require the nSMP.
\end{theorem}

\begin{proof}
Suppose $\A$ is a bipartite sign pattern with a pair of 
intersecting oppositely signed cycles of different lengths.
Then by \cite[Theorem~3.13]{BBCCVV}, $\A$ allows a repeated real eigenvalue. Thus $\A$ allows a matrix $A$ with less than $n$ distinct eigenvalues. 

If the largest composite cycle of $D(\A)$ has length
$n-1$, then $n$ is odd and by \cite[Corollary 2.12(i)]{BBCCVV}, 
$\A$ does not allow a matrix with $n-1$ eigenvalues. 
Likewise, by \cite[Corollary 2.12(ii)]{BBCCVV}, if $\A$ is sign-nonsingular and $n$ is even, then $\A$ does not allow 
$n-1$ eigenvalues. Therefore by Theorem~\ref{thm:bifurcation}, $A$ does not have the nSMP, so $\A$ does not require the nSMP.

If a pattern is sign-nonsingular, then the corresponding digraph has exactly one composite cycle of length $n$. Thus, with the assumptions given, either $D(\A)$ has a composite cycle of length $n$ or $n-1.$  By Corollary~\ref{cor:n-1}, $\A$ allows the nSMP.
\end{proof}

\begin{example}\label{ex:bipartite1}
The digraphs in Figure~\ref{fig:di4} (with the nonsigned arcs considered positive) correspond to  bipartite patterns $\A$ that allow the nSMP but do not require the nSMP by Theorem~\ref{thm:anotr}
since $n =4$ and $\A$ is sign-nonsingular.
For each pattern respectively, the following matrices $A_1$ and $A_2$ are realizations which do not have the nSMP since  $X_1$ and  $X_2$
are corresponding nonzero matrices satisfying the nSMP conditions (\ref{eq:had}),(\ref{eq:commute}),
and (\ref{eq:trace}).
\begin{small}
\[
A_1=\left[\begin{array}{rrrr}
0&1&0&0\\
2&0&-1&0\\
0&0&0&1\\
1&0&0&0
\end{array}\right]\ 
X_1=\left[\begin{array}{rrrr}
1&0&1&0\\
0&1&0&1\\
-1&0&-1&0\\
0&-1&0&-1
\end{array}\right] \
A_2=\left[\begin{array}{rrrr}
0&1&0&0\\
3&0&-1&0\\
0&1&0&1\\
1&0&0&0
\end{array}\right] \
X_2=\left[\begin{array}{rrrr}
-2&0&-4&0\\
0&-1&0&-1\\
1&0&2&0\\
0&1&0&1
\end{array}\right]
\]\end{small}
While it may or may not be straightforward (and it is not necessary by Theorem~\ref{thm:anotr}) to determine matrix realizations that do not have the nSMP, the patterns in Figure~\ref{fig:di4} can be generalized in various ways to $n\times n$ sign patterns with $n$ even,  to obtain patterns that allow but do not require the nSMP. One general example would be to simply replace the $4$-cycle with an $n$-cycle with $n$ even.
\end{example}
\begin{figure}[h!t]
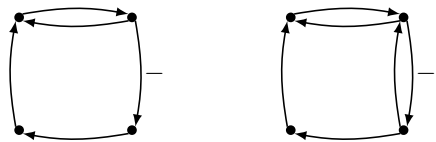

    \begin{center}
\tikzpicture
\node (2) at (0.75,0.75)[place]{};
\node (1) at (-0.75,0.75)[place]{};
\node (4) at (-0.75,-0.75)[place]{};
\node (23) at (1.05, 0)[]{$-$};
\node (3) at (0.75,-0.75)[place]{};
\draw [right] (2.south west) to [bend left=10] (1.south east);
\draw [right] (2.south east) to [bend left=10]  (3.north east);
\draw [right] (3.south west) to [bend left=10]  (4.south east);
\draw [right] (4.north west) to [bend left=10] (1.south west);
\draw [right] (1.north east) to [bend left=10]  (2.north west);
\endtikzpicture \qquad \qquad 
\tikzpicture
\node (2) at (0.75,0.75)[place]{};
\node (1) at (-0.75,0.75)[place]{};
\node (4) at (-0.75,-0.75)[place]{};
\node (23) at (1.05, 0)[]{$-$};
\node (3) at (0.75,-0.75)[place]{};
\draw [right] (2.south west) to [bend left=10] (1.south east);
\draw [right] (2.south east) to [bend left=10]  (3.north east);
\draw [right] (3.north west) to [bend left=10] (2.south west);
\draw [right] (3.south west) to [bend left=10]  (4.south east);
\draw [right] (4.north west) to [bend left=10] (1.south west);
\draw [right] (1.north east) to [bend left=10]  (2.north west);
\endtikzpicture 
\end{center}
\caption{Digraphs of patterns that allow but do not require nSMP.}\label{fig:di4}
\end{figure}

\begin{example}
The digraph in Figure~\ref{fig:di42} is a bipartite digraph satisfying the second condition of Theorem~\ref{thm:anotr} 
and hence the corresponding pattern will allow but not require the nSMP. This digraph can be easily extended to larger patterns by replacing the $4$-cycle by an $(n-1)$-cycle for odd $n\geq 7$. Another option is to attachd $2$-cycles to three or four of the vertices of the the $4$-cycle in Figure~\ref{fig:di42} to create other patterns which allow but not require the nSMP (the case with four $2$-cycles would be a sign-nonsingular example). 
\end{example}

\begin{figure}[h!t]
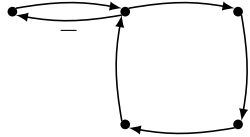

    \begin{center}
\tikzpicture
\node (5) at (-2.25,0.75)[place]{};
\node (2) at (0.75,0.75)[place]{};
\node (1) at (-0.75,0.75)[place]{};
\node (4) at (-0.75,-0.75)[place]{};
\node (15) at (-1.5, 0.50)[]{$-$};
\node (3) at (0.75,-0.75)[place]{};
\draw [right] (5.north east) to [bend left=10] (1.north west);
\draw [right] (1.south west) to [bend left=10] (5.south east);
\draw [right] (2.south east) to [bend left=10]  (3.north east);
\draw [right] (3.south west) to [bend left=10]  (4.south east);
\draw [right] (4.north west) to [bend left=10] (1.south west);
\draw [right] (1.north east) to [bend left=10]  (2.north west);
\endtikzpicture 
\end{center}
\caption{Digraph of a pattern that allows but does not require the nSMP.}\label{fig:di42}
\end{figure}

The patterns in Examples~\ref{ex:starloops} and \ref{ex:starthree} can also be extended to create larger classes of graphs that allow but do not require the nSMP, by attaching a star to an appropriate digraph that allows the nSMP.

\begin{example}
Let $\G$ be a sign pattern on $(n-2)$ vertices that has a composite $(n-3)$-cycle. 
Choose any vertex $v$ of $D(\G)$ such that there is a composite $(n-3)$-cycle  $C$ on the remaining vertices of $D(\G)$. Without loss of generality, we may assume
that $v$ is the first vertex of $D(\G)$. 
Attach a pair of oppositely signed $2$-cycles to $v$ to obtain a sign pattern 
$\A$ with $D(\A)$ as the first digraph in Figure~\ref{fig:startwothree}. Then $\A$ allows the nSMP by Corollary~\ref{cor:n-1} since $D(\A)$ has a
composite $(n-1)$-cycle consisting of $C$ and a $2$-cycle. By choosing 
$A\in Q(\A)$ and $X\neq \O$
as in Example~\ref{ex:starloops}, then $A\circ X=\O$, $X^TA-AX^T=\O$ and
${\rm{tr}}(X^TA)=0$ for $0\leq k<n$. Thus $A$ does not have the nSMP. Therefore, $\A$ allows the nSMP, but $\A$ does not require the nSMP.
\end{example}

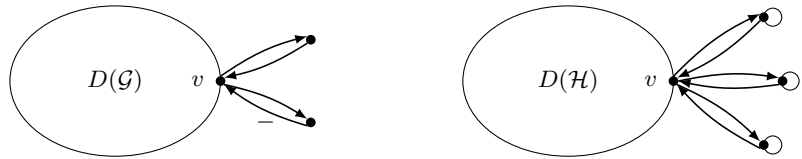
\begin{figure}[ht]
	\begin{center}
	\begin{small}	
		\begin{tikzpicture}
        \draw (0,0) ellipse (14mm and 10mm) {};
		\node at (0,0) {$D(\G)$};        
         \node[place] (1) at (1.4,0){};    
         \node at (1.1,0){$v$};
         \node[place] (3) at (2.6,-0.55){};       \node[place] (2) at (2.6,0.55){};
         \node (13) at (2, -0.55)[]{$-$};
         \draw [right] (1.north) to [bend left=10] (2.north west);
         \draw [right] (2.south west) to [bend left=10] (1.north east);
         \draw [right] (1.south east) to [bend left=10] (3.north west);
         \draw [right] (3.south west) to [bend left=10] (1.south east);
         \end{tikzpicture}\qquad\qquad\qquad 
         \begin{tikzpicture}
        \draw (0,0) ellipse (14mm and 10mm) {};
		\node at (0,0) {$D(\cH)$};   
            \node at (1.1,0){$v$};
		  \node[place] (1) at (1.4,0){};        
         \node[place] (2) at (2.6,0.85){};        
         \node[place] (3) at (2.85,0){};
         \node[place] (4) at (2.6,-0.85){};
         \draw (2.72,0.85) circle (1.2mm) {};
         \draw (2.95,0) circle (1.2mm) {};
         \draw (2.72,-0.85) circle (1.2mm) {};
         \draw [right] (1.north) to [bend left=10] (2.north west);
         \draw [right] (2.south west) to [bend left=10] (1.north east);
         \draw [right] (1.east) to [bend left=10] (3.north west);
         \draw [right] (3.south west) to [bend left=10] (1.east);
         \draw [right] (1.south east) to [bend left=10] (4.north west);
         \draw [right] (4.south west) to [bend left=10] (1.south east);
         \end{tikzpicture} 
      \end{small}  
   \end{center}
   \caption{Digraph constructions for patterns that allow but do not require the nSMP.}\label{fig:startwothree}
	\end{figure}

\begin{example}
    Let $\cH$ be any sign pattern on $(n-3)$ vertices that has a composite cycle $C$ of length $(n-4)$ or $(n-3)$. Choose any vertex $v$ of $D(\cH)$. Without loss of generality, we may assume $v$ is the first vertex of $D(\cH)$. Attach three positively signed $2$-cycles to $v$ with positively signed loops attached to the three new vertices, to obtain a sign
    pattern $\A$ with $D(\A)$ as the second digraph in Figure~\ref{fig:startwothree}. Then $\A$ allows the nSMP by Corollary~\ref{cor:n-1} since $C$ along with the three loops, is a composite cycle of length $(n-1)$ or $n$ in $D(\A)$. By choosing $A\in Q(\A)$ and $X\neq \O$ as in Example~\ref{ex:starthree},
    then $A\circ X=\O$, $X^TA-AX^T=\O$ and
${\rm{tr}}(X^TA)=0$ for $0\leq k<n$. Thus $A$ does not have the nSMP. Therefore $\A$ allows the nSMP, but does not require the nSMP.
    \end{example}

An $n\times n$ sign pattern  $\A$ is \emph{spectrally arbitrary} if for any set of $n$ complex numbers closed under conjugation, there is a matrix $A\in Q(\A)$ that has the set
of complex numbers as its spectrum. In other words, the combinatorial sign structure of the pattern does not restrict the possible eigenvalues of a matrix with the pattern if the pattern is spectrally arbitrary. Every spectrally 
arbitrary sign pattern allows the nSMP by Theorem~\ref{thm:hessy}.
Given that the nSMP property is used to give insight into eigenvalue multiplicities of a sign pattern, one might expect that a spectrally arbitrary pattern would require the nSMP. However, this
is not the case, as the next example demonstrates.

\begin{example}\label{ex:sap}
Let 
\[ 
\A=\begin{bmatrix}
+&+&0&0&0\\
0&0&+&+&+\\
-&-&0&0&0\\
0&0&-&0&0\\
0&0&0&-&-\\
\end{bmatrix},\  
A=\left[\begin{array}{rrrrr}
1&1&0&0&0\\
0&0&1&1&1\\
-1&-1&0&0&0\\
0&0&-1&0&0\\
0&0&0&-1&-b\\
\end{array}\right]
{\rm {\ and \ }}
X=\left[ \begin{array}{rrrrr}
0&0&0&-1&1\\
0&0&0&0&0\\
0&0&0&-1&1\\
0&0&0&0&0\\
0&0&0&0&0\\
\end{array}\right].
\]
Then $\A$ is spectrally arbitrary by \cite[Theorem~4.7]{CFV} since
$\A=\mathcal{C}_{5,2}$ in \cite{CFV}. Note that $A\in Q(\A)$ for every positive $b\in \R$.
If $b=2$, then $A$ has the nSMP. But if
$b=1$, then $A$ does not have the nSMP since in this case, the displayed $X\neq \O$ satisfies $A\circ X=\O$, $AX^T=\O=X^TA$, and ${\rm{tr}}(X^TA^k)=0$ for $0\leq k\leq 4$. Thus $\A$ is a spectrally arbitrary pattern that allows the nSMP but does not require the nSMP. 
\end{example}

Example~\ref{ex:sap}  generalizes to all the patterns $\cC_{n,2}$ with odd $n\geq 5$ (as in \cite{CFV}) which are spectrally arbitrary. 
Taking $A\in Q(\cC_{n,2})$ with entries having absolute value $1$ and with $X$ having the same $3\times 2$ nonzero block in the right corner as above, will demonstrate that $\cC_{n,2}$ does not require the nSMP but does allow the nSMP.

\section{Small patterns and reducible patterns}\label{sec:small}

In this section we describe the $n\times n$  sign patterns with respect to the nSMP for $n\leq 3$. We also make some remarks about reducible patterns.
Note that a $1\times 1$ pattern $\A$ always requires the nSMP since 
$X=[0]$ if ${\rm{tr}}(X)=0$.

\begin{theorem}\label{thm:order2}
Suppose $\A$ is a $2\times 2$ sign pattern. 
If $\A$ is irreducible,  
then $\A$ requires the nSMP. If $\A$ is equivalent to a reducible pattern and the diagonal of $\A$ has two distinct symbols from $\{0,+,-\}$, then $\A$ requires the nSMP. If $\A$ is reducible and the two diagonal entries are zero, then 
$\A$ does not allow the nSMP, but if the two diagonals are equal but nonzero, then $\A$ allows the nSMP, but does not require the nSMP.
\end{theorem}    

\begin{proof}
Suppose $\A$ is a $2\times 2$ pattern. If $\A$ is irreducible, then $D(\A)$ is a $2$-cycle with or without loops, and so $\A$ requires the nSMP by Theorem~\ref{thm:cycle}. If $\A$ is reducible, the result follows from Theorem~\ref{thm:reducible}.
\end{proof}

To efficiently view the pattern options, we use a 
relaxation of the notion of sign pattern (as introduced in \cite{CF}). Let $\rr=\{+,-,0,*,\oplu,\omin,\ostar\}$.
An \emph{$\rr$-pattern} is a matrix with entries in $\rr$.
The definition of $Q(\A)$ for a sign pattern $\A$ extends directly to a $\rr$-pattern. 
An $\rr$-pattern $\B$ is a \emph{relaxation} of an $\rr$-pattern $\A$ if $Q(\A)\subseteq Q(\B)$.
If $\B$ is an $\rr$-pattern, then $\A$ is a {\it fixed signing} of $\B$ if  $\B$ is a relaxation of $\A$ and $\A$ is a sign pattern. 

The results of Theorem~\ref{thm:order2} can now be described as follows:
A $2\times 2$ sign pattern requires the nSMP if and only if it is equivalent
to a fixed signing of
\[
\left[
\begin{array}{cc}
\ostar&*\\
*&\ostar
\end{array}
\right]
\quad \rm{\ or\ }
\quad
\left[
\begin{array}{cc}
+&\ostar\\
0&\omin
\end{array}
\right],
\]
depending on whether it is irreducible or not. And a $2\times 2$ pattern
allows, but does not require the nSMP if and only if it is equivalent to
a fixed signing of 
\[
\left[
\begin{array}{cc}
+&\ostar\\
0&+
\end{array}
\right].
\]
Finally, a $2\times 2$ sign pattern does not allow the nSMP if and only
if it is equivalent to a fixed signing of
\[
\left[
\begin{array}{cc}
0&\ostar\\
0&0
\end{array}
\right],
\]
by Corollary~\ref{cor:n-1}.

\begin{theorem}\label{thm:order3}
Suppose $\A$ is an irreducible $3\times  3$ sign pattern.
Then $\A$ requires the nSMP, unless $D(\A)$ is a star pattern with no non-centre loops and with oppositely signed $2$-cycles. In the latter case, $\A$ allows the nSMP but does not require the nSMP.
\end{theorem}

\begin{proof} 
Suppose $\A$ is an irreducible $3\times 3$ sign pattern. 
If all the off-diagonal entries are nonzero, then $\A$ requires the nSMP by
Theorem~\ref{thm:hollow}.
If $D(\A)$ has no $3$-cycle, then $D(\A)$ is a star pattern and so the result follows from Theorem~\ref{thm:star} and Corollary~\ref{cor:n-1}.
If $D(\A)$ has no $2$-cycle
then $D(\A)$ is a $3$-cycle
and so $\A$ requires the nSMP by Theorem~\ref{thm:cycle}.
Suppose $D(\A)$ has a $3$-cycle and a $2$-cycle, but not every off-diagonal entry of $\A$ is nonzero. Then by permutation similarity, we may 
assume $\A_{1,3}=0$ and $\A_{1,2}\neq 0$, $\A_{2,3}\neq0$ and $\A_{3,1}\neq 0$. Let $A\in Q(\A)$ and $A\circ X=\O$ with $AX^T=X^TA$ and
${\rm{tr}}(X^T)=0$. Note that $A$ is a Hessenberg matrix and so by Theorem~\ref{rem:Hessy}, $X^T=c_0I+c_1A+c_2A^2$ for
some $c_0,c_1,c_2\in \R$. Since $(A^2)_{1,3}\neq 0$,
but $(X^T)_{13}=A_{1,3}=0$, $c_2=0$.  Thus $X=c_0I+c_1A^T$. Now $\O=A\circ X= A\circ (c_0I+c_1A^T)$. Thus, for $i\neq j$,
$c_1(A\circ A^T)_{ij}=0$. Since $D(\A)$ has at least
one $2$-cycle, $c_1=0$. Then $X^T=c_0I$. Hence $c_0=0$ by the trace condition. Therefore $X=\O$ and 
$\A$ requires the nSMP.
\end{proof}

The following remark follows  from Theorem~\ref{thm:reducible}.

\begin{remark}\label{rem:oplus1} Suppose $\A$ is a reducible pattern with one diagonal block $[0]$. Then $\A$ requires the nSMP if and only if $\A$ is equivalent to a superpattern of $\cS\oplus [0]$ where $\cS$ is a pattern that requires the nSMP and is sign-nonsingular. 
\end{remark}

In \cite[Theorem 1.10]{EJ1} it was observed that an $n \times n$ sign pattern $\A$ allows a positive real eigenvalue if and only if $D(\A)$ has has a positive simple cycle. 
Combining this observation with Theorem~\ref{thm:reducible}, we obtain the following remark.

\begin{remark}\label{rem:oplus}
Suppose $\A$ is a reducible pattern with one diagonal block $[+]$. Then $\A$ requires the nSMP if and only if $\A$ is equivalent to a  superpattern of $\cS \oplus [+]$ where $\cS$ is a pattern that requires the nSMP and $D(\cS)$ does not have a positive simple cycle.   
\end{remark}

By taking a sign pattern $\cS$ that requires the nSMP and requires a positive
eigenvalue, one can construct further patterns that allow
but do not require the nSMP.

\begin{example}
Let 
\[\cS =\left[
\begin{array}{ccc}
+&-&+\\+&0&-\\+&+&0\\
\end{array}
\right].
\]
Then $\cS$ requires a positive eigenvalue by \cite[Lemma 5.1]{KMT}. Further $\cS$ 
requires the nSMP by Theorem~\ref{thm:hollow}. Hence $\cS \oplus [+]$ allows the
nSMP but does not require then nSMP by Remark~\ref{rem:oplus} since, for any $B\in Q(\cS),$ the positive diagonal block $[+]$ can be scaled to either match a positive eigenvalue of $B$ or not. 
\end{example}

Using Remarks~\ref{rem:oplus1} and \ref{rem:oplus}, the following provides a characterization of the reducible $3\times 3$ sign patterns that allow the nSMP and those that require the nSMP.

\begin{theorem}
    A reducible $3\times 3$ sign pattern $\A$ requires the nSMP if and only if $\A$ is equivalent to a fixed signing of a pattern in Figure~\ref{fig:requires}.
\begin{figure}
\[
\begin{array}{cccccc}
\left[ 
\begin{array}{cc|c}
\ostar&*&\ostar\\
*&0&\ostar\\ \hline
0&0&0
\end{array}
\right] &
\left[ 
\begin{array}{cc|c}
+&\oplu&\ostar\\
+&-&\ostar\\ \hline
0&0&0
\end{array}
\right]
&
\left[ 
\begin{array}{cc|c}
+&+&\ostar\\
-&+&\ostar\\ \hline
0&0&0
\end{array}
\right]
&
\left[ 
\begin{array}{cc|c}
\omin&+&\ostar\\
-&\omin&\ostar\\ \hline
0&0&+
\end{array}
\right]&
\left[ 
\begin{array}{cc|c}
-&\ostar&\ostar\\
0&0&\ostar\\ \hline
0&0&+
\end{array}
\right]&
\left[ 
\begin{array}{cc|c}
0&\ostar&\ostar\\
0&-&\ostar\\ \hline
0&0&+
\end{array}
\right]\\
\A_1&\A_2&\A_3&\A_4&\A_5&\A_6
\end{array}
\]
\caption{Patterns whose fixed signings require the nSMP}\label{fig:requires}
\end{figure}
Further, a reducible $3\times 3$ sign pattern $\A$ allows but does not require the nSMP if and only if $\A$ is equivalent to a fixed signing of 
a pattern in Figure~\ref{fig:allow}.
\begin{figure}
\[
\begin{array}{cccccc}
\left[ 
\begin{array}{cc|c}
+&+&\ostar\\
+&+&\ostar\\ \hline
0&0&0
\end{array}
\right] &
\left[ 
\begin{array}{cc|c}
+&+&\ostar\\
-&-&\ostar\\ \hline
0&0&0
\end{array}
\right] &
\left[ 
\begin{array}{cc|c}
+&+&\ostar\\
-&\ostar&\ostar\\ \hline
0&0&+
\end{array}
\right] &
\left[ 
\begin{array}{cc|c}
\ostar&+&\ostar\\
+&\ostar&\ostar\\ \hline
0&0&+
\end{array}
\right] &
\left[ 
\begin{array}{cc|c}
+&\ostar&\ostar\\
0&\ostar&\ostar\\ \hline
0&0&+
\end{array}
\right]&
\left[ 
\begin{array}{cc|c}
\omin&\ostar&\ostar\\
0&+&\ostar\\ \hline
0&0&+
\end{array}
\right] \\
\B_1&\B_2&\B_3&\B_4&\B_5&\B_6
\end{array}
\] 
\caption{Patterns whose fixed signings allow but do not require the nSMP} \label{fig:allow}
\end{figure}
\end{theorem}

\begin{proof}
Suppose $\A$ is a reducible $3\times 3$ sign pattern. By permutation similarity, we can assume that
$\A$ has the form
\[ 
\left[ \begin{array}{cc|c}  \multicolumn{2}{c|}{\multirow{2}{*}{$\cS$}} & \ostar \\ 
                        \multicolumn{2}{c|}{}&\ostar \\ \hline 
                        0&0&w
                        \end{array}
                        \right]
                        \]
with $w\in \{ 0,+,-\}$. By negation, we may assume that $w\in \{ 0, +\}$.
If $\A$ is a fixed signing of a pattern in Figures~\ref{fig:requires} and \ref{fig:allow} (with the exception of $\B_5$ with the (2,2) entry positive) the pattern $\cS$ requires
the nSMP either by Theorem~\ref{thm:cycle} since $D(\cS)$ is a $2$-cycle, or 
by Theorem~\ref{thm:reducible} since $\cS$ is reducible with distinct signs
on the diagonal.
Note that the pattern $\B_5$ with the (2,2) entry positive 
will allow any three positive real numbers as eigenvalues and so will allow, but not require the nSMP by Theorem~\ref{thm:reducible}.

We now determine the sign patterns that require the nSMP.
Suppose $w=+$. By Remark~\ref{rem:oplus}, $\A$ requires the nSMP if and only if $\cS$ requires the
nSMP and $D(\cS)$ does not have a positive simple cycle. If $D(\cS)$ does not have a $2$-cycle, then $\cS$ is a reducible pattern
with distinct non-positive  diagonal entries
and so $\A$ is equivalent to fixed signing of $\A_5$ or $\A_6$.  
If $D(\cS)$ has a $2$-cycle, then $\A$ is equivalent
to a fixed signing of $\A_4$. 

Suppose $w=0$. By Remark~\ref{rem:oplus1}, $\A$ requires the nSMP if and only if
$\cS$ requires the nSMP and $\cS$ is sign-nonsingular. 
If $\cS$ is sign-nonsingular, then either (1) $D(\cS)$ has a $2$-cycle and at most one loop, or (2) $D(\cS)$
has two loops and no $2$-cycle, or (3) $D(\cS)$ has two loops and a $2$-cycle with the sign of the $2$-cycle
the negative of the product of the signs of the loops. In case (1), $\A$ is equivalent to
a fixed signing of $\A_1$. In case (2), $\A$ is a equivalent to a fixed signing of $\A_2$ with a zero in position (1,2).
In case (3), $\A$ is equivalent to a fixed signing of $\A_2$ if the $2$-cycle is positive and
is equivalent to a fixed signing of $\A_3$ if the $2$-cycle is negative (using the
negation of $\A$ if the two loops are negative). Therefore, the fixed signings of the patterns in Figure~\ref{fig:requires} are precisely the reducible $3\times 3$ sign patterns that require the nSMP.

Next consider the patterns $\A$ that allow but do not require the nSMP.
First observe that every fixed signing of a pattern in Figure~\ref{fig:allow} 
allows the nSMP by Corollary~\ref{cor:n-1} since the digraph of each contains
a composite $2$-cycle. 

If $\cS$ is reducible and $\A$ allows a matrix with a repeated eigenvalue, then $\A$ is equivalent to a fixed signing of $\B_5$ or $\B_6$. These patterns also allow matrices with distinct eigenvalues.
By Theorem~\ref{thm:reducible}, 
$\B_5$ and $\B_6$ will allow the nSMP, but will not require the nSMP. 

Suppose $w=0$. By Remark~\ref{rem:oplus1}, $\cS$ must allow singularity if $\A$ does not require the nSMP. If $\cS$ allows singularity and $\cS$ is reducible, then $\A$ would not allow the nSMP by Theorem~\ref{thm:reducible}. 
Thus $D(\cS)$ has a $2$-cycle. In this case, $\cS$ has no zero entries. Since the rows of any matrix in $Q(\cS)$ must allow
linear dependence to allow singularity, $\A$ must be equivalent to a fixed signing of either $\B_1$ or $\B_2$ (noting
that if the two rows of $\cS$ both have alternating signs, then by signature similarity, the rows will no longer have alternating signs).

Suppose $w=+$. By Remark~\ref{rem:oplus}, $\cS$ must have a positive simple cycle
if $\A$  does not require the nSMP. If $\cS$ has no $2$-cycle, then $\A$ is equivalent
to a fixed signing of $\B_5$ or $\B_6$. If $\cS$ has a positive $2$-cycle, then $\A$ is equivalent
to a fixed signing of $\B_4$ (using signature similarity if both arcs of the $2$-cycle are negative). 
If $\cS$ has a negative $2$-cycle, then $\cS$ must have a positive loop and $\A$ 
is equivalent to a fixed signing of $\B_3$. 

Therefore, the fixed signings  of the patterns in Figure~\ref{fig:allow} are precisely the reducible $3\times 3$ sign patterns that allow but do not require the nSMP.
\end{proof}

\section{Concluding comments}\label{sec:conclude}

In this paper, we determined exactly which patterns do not allow the nSMP, and those that do
allow the nSMP, in terms of a simple feature of the digraph of the pattern. 
We also provide several classes of patterns that require the nSMP, as well as classes of patterns that allow but do not require the nSMP. In the future, further work can be done to determine features that demonstrate a pattern requires the nSMP. Future work could also explore characteristics of matrices that do not have the nSMP when the matrix has a pattern that allows the nSMP. Finally, as per Remark~\ref{rem:oplus1}, it would be of interest to explore which sign-nonsingular patterns require the nSMP. Example~\ref{ex:bipartite1} demonstrates that not every sign-nonsingular pattern requires the nSMP.

A corresponding problem is the symmetric version, when a \emph{graph} requires the strong multiplicity property (see e.g. \cite{bifurcate, LOS}). The results are different for such a case due to the extra constraint of matrix symmetry and the added restriction that $X\circ I=\O$ for the graph version. The investigation of when a graph requires the strong multiplicity property was first approached in \cite{LOS}.

\end{document}